\documentclass[reqno]{amsart}
\usepackage{amsmath,amssymb,amsthm,mathtools}
\usepackage[a4paper,left=2.6cm,right=2.6cm,top=2cm,bottom=2cm]{geometry}
\usepackage{float}
\usepackage[table]{xcolor}
\usepackage[section]{placeins}
\usepackage{listings}
\lstdefinestyle{sagestyle}{
    basicstyle=\ttfamily\footnotesize,
    breakatwhitespace=false,
    breaklines=true,
    captionpos=b,
    keepspaces=true,
    numbers=left,
    numberstyle=\tiny,
    numbersep=5pt,
    showspaces=false,
    showstringspaces=false,
    showtabs=false,
    tabsize=4
}
\usepackage{hyperref}
\hypersetup{
    colorlinks=true,
    linkcolor=blue,
    citecolor=blue,
    urlcolor=blue
}

\theoremstyle{plain}
\newtheorem{theorem}{Theorem}[section]
\newtheorem{lemma}[theorem]{Lemma}
\newtheorem{proposition}[theorem]{Proposition}
\newtheorem{corollary}[theorem]{Corollary}
\theoremstyle{definition}
\newtheorem{definition}[theorem]{Definition}
\newtheorem{notation}[theorem]{Notation}
\newtheorem{example}[theorem]{Example}
\theoremstyle{remark}
\newtheorem{remark}[theorem]{Remark}

\theoremstyle{plain}
\newtheorem{thmx}{Theorem}

\newcommand{\C}{\mathbb{C}}
\newcommand{\Z}{\mathbb{Z}}
\DeclareMathOperator{\sgn}{sgn}

\newcommand{\vc}[2]{\begin{bmatrix}#1\\#2\end{bmatrix}}
\newcommand{\vcc}[3]{\begin{bmatrix}#1\\#2\\#3\end{bmatrix}}

\title{Stability of plethysm coefficients and modified polynomial induction}
\author{Soumyadip Sarkar}
\address{$^{1}$The Institute of Mathematical Sciences, CIT Campus, Taramani, Chennai 600113, India}
\address{$^{2}$Homi Bhabha National Institute, Training School Complex, Anushakti Nagar, Mumbai 400094, India}
\email[Soumyadip Sarkar]{soumyadips@imsc.res.in}
\subjclass[2020]{05E05, 05E10, 20C30}
\keywords{Symmetric functions, plethysm, vector-partition functions, Foulkes coefficient, skewing operators}

\begin{document}
\begin{abstract}
Let $h_n$ denote the complete homogeneous symmetric function of degree $n$,
$s_\nu$ the Schur function indexed by a partition $\nu$, and
$\langle\cdot,\cdot\rangle$ the Hall inner product. The
plethysm coefficient $\langle h_n[h_m], s_\nu\rangle$ is the multiplicity of
the Weyl module $W_\nu(\mathbb{C}^N)$ in $\mathrm{Sym}^n(\mathrm{Sym}^m(\mathbb{C}^N))$ whenever
$\nu$ has at most $N$ parts. For a partition $\lambda$, let
$\lambda[mn]=(mn-|\lambda|,\lambda_1,\lambda_2,\dots)$. Bowman and Paget
proved that $\langle h_n[h_m], s_{\lambda[mn]}\rangle$ is independent of $m$
and $n$ once $m,n\ge|\lambda|$.

We prove that this coefficient is already constant for $m\ge\lambda_1$ and
$n\ge\lfloor(|\lambda|+\lambda_1)/2\rfloor$, and that its stable value equals
$\langle s_\lambda, H[H-1-h_1]\rangle$, where $H=\sum_{i\ge0}h_i$; the proof
uses the skewing operators adjoint to multiplication by complete homogeneous symmetric functions. We give a short
new proof of the Bowman--Paget theorem based on stability of vector-partition
functions, which count the ways of writing a vector of non-negative integers as
an unordered sum of non-zero such vectors (its parts). Using a Vandermonde-type
determinantal identity, we deduce that the stable value is a signed sum of
numbers of vector partitions all of whose parts have entries summing to at
least~$2$. The same method gives a short proof of Brion's
theorem that $\langle h_n[h_{m+d}], s_{\lambda+(nd)}\rangle$ is eventually
constant in $d$, where $\lambda+(nd)$ is obtained from $\lambda$ by adding
$nd$ to its first part. We show that the stable value equals the multiplicity
of $W_\lambda(\mathbb{C}^k)$ in an explicit representation of $GL_k(\mathbb{C})$, where
$k=\lfloor|\lambda|/2\rfloor$, independent of $m$ and $n$: the space of
homogeneous polynomial functions of degree $|\lambda|$ on $k\times k$ matrices
that are invariant under permuting rows and have no monomial of degree $1$ in
the entries of any single row, with $GL_k(\mathbb{C})$ acting by right translation.
This module is a modification of the polynomial induction of Narayanan, Paul,
Prasad and Srivastava, which constructs polynomial representations of general
linear groups from representations of symmetric groups.
Finally, we prove an explicit positive formula for the stable value, as the
number of lattice points of a polytope, when $\lambda_1\le 4$.

\end{abstract}
\maketitle
\vspace{-1em}
\tableofcontents

\section*{Introduction}

The composition of two polynomial representations of $GL_N(\C)$ is reflected
by the plethysm operation on the ring of symmetric functions. The plethysm of
two Schur functions expands in the Schur basis as
$s_\lambda[s_\mu]=\sum_{\nu}c_{\lambda\mu}^\nu s_\nu$, and the coefficients
$c_{\lambda\mu}^\nu$ are known as \emph{plethysm coefficients}. Finding a
combinatorial interpretation of plethysm coefficients is a well-known open
problem in algebraic combinatorics \cite{colmenarejo}. Partial results
include a description of the hook and near-hook constituents
\cite{langley-remmel}, the lexicographically largest term of a plethysm of
Schur functions \cite{ossz-quasi}, and obstructions to expressing plethysm
coefficients as counting functions of lattice points in polytopes
\cite{kahle-michalek}. Throughout,
$\langle\cdot,\cdot\rangle$ denotes the Hall inner product, for which the
Schur functions are orthonormal, so that $\langle h_n[h_m], s_\nu\rangle$ is
the coefficient of $s_\nu$ in $h_n[h_m]$. Since the character of
$\mathrm{Sym}^n(\mathrm{Sym}^m(\C^N))$ is $h_n[h_m](x_1,\dots,x_N)$, this
coefficient is the multiplicity of the Weyl module $W_\nu(\C^N)$ in
$\mathrm{Sym}^n(\mathrm{Sym}^m(\C^N))$ for every $N\ge \ell(\nu)$. For
background on plethysm and symmetric functions we refer to
\cite[p.~135]{macdonald} and \cite{stanley}.

\medskip
\noindent\textbf{Stability.} To better understand plethysm coefficients, many
authors have studied their stability properties. Foulkes \cite{foulkes}
conjectured that for all $m$, $n$ and $\lambda$, the coefficient
$\langle h_n[h_{m+d}],s_{\lambda+(nd)}\rangle$ becomes constant as
$d\to\infty$; this was proved by Brion \cite{brion} using algebro-geometric
methods; Colmenarejo \cite{colmenarejo-stability} later gave a combinatorial
proof, decomposing plethysm coefficients into alternating sums of numbers of
integer points in polytopes. Foulkes \cite{foulkes} also conjectured that $h_n[h_m]-h_m[h_n]$ is
Schur-positive for $n\ge m$, a conjecture that remains open. For a partition
$\lambda$, write $\lambda[mn]=(mn-|\lambda|,\lambda_1,\lambda_2,\dots)$.
Using the representation theory of partition algebras, Bowman and Paget
\cite{bow-pag} proved that $\langle h_n[h_m],s_{\lambda[mn]}\rangle$ is
independent of $m$ and $n$ for $m,n\ge|\lambda|$; in particular,
$\langle h_n[h_m],s_{\lambda[mn]}\rangle=\langle h_m[h_n],s_{\lambda[mn]}\rangle$
in this range. We call this common value the \emph{stable Foulkes
coefficient} of $\lambda$. Further stability results for plethysm
coefficients were obtained by Law and Okitani \cite{law-okitani}, who
verified a conjecture of Wildon and a stability prediction of Bessenrodt,
Bowman and Paget, and Paget and Wildon \cite{paget-wildon} recently proved two
general stability theorems for plethysms of Schur functions that recover many
earlier results.

\medskip
\noindent\textbf{Our results.} We first observe that the plethysm
coefficients of $h_n[h_m]$ are signed sums of vector-partition functions, and
that simple stability properties of vector partitions give short proofs of
the theorems of Bowman--Paget and of Brion. This yields a formula for the
stable Foulkes coefficient as a signed sum of vector-partition functions with
all parts of size at least~$2$ (Theorem~\ref{thmA}), which we identify with
$\langle s_\lambda, h_{\lfloor|\lambda|/2\rfloor}[H-h_1]\rangle$, where
$H=\sum_{i\ge0}h_i$, via an identity between signed sums of vector-partition
functions (Theorem~\ref{thmB}).

We then interpret the stable value representation-theoretically. Narayanan,
Paul, Prasad and Srivastava \cite{narayanan} constructed a \emph{polynomial
induction} functor from representations of $S_n$ to polynomial
representations of $GL_n(\C)$, right adjoint to restriction, and used it to
give a representation-theoretic proof of Littlewood's plethystic formula for
restriction coefficients. We modify this construction to realize the stable
Foulkes coefficient as the multiplicity of $W_\lambda$ in an explicit
$GL_{\lfloor|\lambda|/2\rfloor}(\C)$-module that does not depend on $m$ or $n$
(Theorem~\ref{thmC}).

Finally, using the action of the skewing operators $h_r^\perp$, we extend the
range of stability in the theorem of Bowman and Paget to $m\ge\lambda_1$ and
$n\ge\lfloor(|\lambda|+\lambda_1)/2\rfloor$, and show that the stable value
equals $\langle s_\lambda, H[H-1-h_1]\rangle$ (Theorem~\ref{thmD}). For
example, for $\lambda=(2,2)$ the coefficient is already stable for $m\ge2$ and
$n\ge3$, rather than $m,n\ge4$ (see Table~\ref{table-22}).

\medskip
\noindent\textbf{Main results.} In what follows, for a $GL_N(\C)$-module $V$,
we write $\langle W_\nu, V\rangle$ for the multiplicity of $W_\nu(\C^N)$ in $V$. Let $S_n$ denote the group of permutations on $n$ letters, and $\mathrm{sgn}(\sigma)$ the sign of a permutation $\sigma$.

\begin{thmx}\label{thmA}
The double sequence $\{\langle h_n[h_m],s_{\lambda[mn]}\rangle\}_{m,n\ge1}$ is
constant for $m,n\geq |\lambda|$, with value
$$\sum_{\sigma\in S_{\ell(\lambda)}}\sgn(\sigma)\,
p^{\geq 2}\big(\lambda_1+\sigma(1)-1,\lambda_2+\sigma(2)-2,\dots,
\lambda_{\ell(\lambda)}+\sigma(\ell(\lambda))-\ell(\lambda)\big),$$
where $p^{\ge 2}(\alpha)$ is the number of vector partitions of $\alpha$ with
all parts of size at least~$2$ (Notation~\ref{notation-p-geq-2}).
\end{thmx}

\noindent Theorem~\ref{thmA} follows from Theorem~\ref{proof1} and
Theorem~\ref{equality-of-vector-partitions}. Its formula resembles the formula of
\cite[Corollary~4.2]{narayanan} for a restriction coefficient:
$$\dim\big(W_\lambda(\C^n)^{S_n}\big)=\sum_{\sigma\in S_{n}}\sgn(\sigma)\,
p_n\big((\lambda_i+\sigma(i)-i)_{1\le i\le n}\big).$$
Our proof of the Bowman--Paget theorem first expresses the stable value as a
signed sum over $S_{\ell(\lambda)+1}$ of the full vector-partition function
$p(\cdot)$. The passage to the signed sum of $p^{\ge2}(\cdot)$ over
$S_{\ell(\lambda)}$ in Theorem~\ref{thmA}, which corresponds to removing $h_1$
from $H$, is given by the following identity.

\begin{thmx}\label{thmB}
For every vector $\lambda=(\lambda_1,\dots,\lambda_k)$,
\begin{align*}
&\sum_{\sigma\in S_{k+1}}\sgn(\sigma)\,
p\big(\lambda_1+\sigma(2)-2,\lambda_2+\sigma(3)-3,\dots,\lambda_k+\sigma(k+1)-(k+1)\big)\\
&\qquad=\sum_{\sigma\in S_k}\sgn(\sigma)\,
p^{\geq 2}\big(\lambda_1+\sigma(1)-1,\lambda_2+\sigma(2)-2,\dots,\lambda_k+\sigma(k)-k\big),
\end{align*}
where $p(\cdot)$ is the vector-partition function (Notation~\ref{notation-vec-par}).
\end{thmx}

\noindent Theorem~\ref{thmB} is proved as
Theorem~\ref{equality-of-vector-partitions}.

\begin{thmx}\label{thmC}
For all $m,n\ge|\lambda|$ and $N\ge\ell(\lambda)+1$,
$$\langle W_{\lambda[mn]},\mathrm{Sym}^n(\mathrm{Sym}^m(\C^N))\rangle
=\langle W_\lambda(\C^{\lfloor|\lambda|/2\rfloor}),
\mathrm{Ind}_{\hat{1}}^{|\lambda|}\ _{S_{\lfloor|\lambda|/2\rfloor}}^{GL_{\lfloor|\lambda|/2\rfloor}}\mathbf{1}\rangle,$$
where the modified induction on the right-hand side is defined in
Equation~\eqref{def-modified-poly-ind}.
\end{thmx}

\noindent Theorem~\ref{thmC} is proved as Theorem~\ref{final-corollary}.

\begin{thmx}\label{thmD}
For $m\ge\lambda_1$ and $n\ge\lfloor(|\lambda|+\lambda_1)/2\rfloor$, whenever
$\lambda[mn]$ is a partition,
$$\langle s_{\lambda[mn]}, h_n[h_m]\rangle=\langle s_\lambda,H[H-1-h_1]\rangle$$ where $H=\sum\limits_{i\ge 0 }h_i$.
\end{thmx}

\noindent Theorem~\ref{thmD} is proved as Theorem~\ref{thm-perp}.

The closed form in Theorem~\ref{thmD} makes the stable value computable for
partitions with few columns. For a partition $\lambda$, let $m_i=m_i(\lambda)$
denote the number of parts of $\lambda$ equal to $i$, and write
$x_+=\max(x,0)$.

\begin{thmx}\label{thmE}
Let $\lambda=(4^{m_4}3^{m_3}2^{m_2}1^{m_1})$ be a partition with
$\lambda_1\le4$. The stable Foulkes coefficient of $\lambda$ is zero if
$m_1\ge2$ or $m_3\ge2$, and otherwise equals
$$
\begin{cases}
(m_4+1)+(m_4-1)_++(m_4-2)_++m_4\cdot[m_2\ge1] & \text{if } (m_1,m_3)=(0,0),\\[4pt]
m_4+(m_4-1)_+ & \text{if } (m_1,m_3)=(1,0),\\[4pt]
(m_4+1)+(m_4-1)_+ & \text{if } (m_1,m_3)=(0,1),\\[4pt]
m_4 & \text{if } (m_1,m_3)=(1,1),
\end{cases}
$$
where $[m_2\ge1]$ is $1$ if $m_2\ge1$ and $0$ otherwise. It is also the number
of lattice points of an explicit polytope (Theorem~\ref{thm-four-columns}).
\end{thmx}

\noindent In particular, if $\lambda_1\le3$, the stable Foulkes coefficient is
$1$ when $\lambda=(2^n)$ or $\lambda=(3,2^n)$ for some $n\ge0$, and $0$
otherwise (Corollary~\ref{cor-three-columns}). Theorem~\ref{thmE} is proved in
Section~\ref{sec-few-columns}, via a reduction to the representation theory
of $GL_4(\C)$.

\medskip
\noindent\textbf{Organization.} In Section~\ref{sec-1} we express the
plethysm coefficients of $h_n[h_m]$ through vector-partition functions. In
Section~\ref{sec-stability} we give short proofs of the theorems of
Bowman--Paget (Theorem~\ref{proof1}) and Brion (Theorem~\ref{proof2}). In
Section~\ref{stable-plethysm} we prove Theorem~\ref{thmB} and complete the
proof of Theorem~\ref{thmA}. In Section~\ref{section-on-polynomial-rep} we
recall polynomial induction from \cite{narayanan}, and in
Section~\ref{modified-poly-rep} we modify it to prove Theorem~\ref{thmC}. In
Section~\ref{stability-using-perp} we use the action of $h_r^\perp$ to prove
Theorem~\ref{thmD}. In Section~\ref{sec-few-columns} we prove
Theorem~\ref{thmE}. Finally, in Section~\ref{sec-sage} we compute some stable
Foulkes coefficients using SageMath.

\section{Combinatorics of vector-partition functions}\label{sec-1}
A \emph{vector partition} of a vector $\alpha=(\alpha_1,\alpha_2,\ldots,\alpha_l)$
is an (unordered) decomposition of $\alpha$ as a sum of non-zero vectors with
non-negative integer entries. For example, the vector partitions of $(2,2)$ are
\begin{alignat*}{2}
\vc{2}{2} &= \vc{2}{2} && \text{(1 part)}\\[4pt]
&= \vc{2}{0}+\vc{0}{2}
 = \vc{1}{0}+\vc{1}{2}
 = \vc{0}{1}+\vc{2}{1}
 = \vc{1}{1}+\vc{1}{1} &\qquad& \text{(2 parts)}\\[4pt]
&= \vc{2}{0}+\vc{0}{1}+\vc{0}{1}
 = \vc{0}{2}+\vc{1}{0}+\vc{1}{0}
 = \vc{1}{1}+\vc{1}{0}+\vc{0}{1} && \text{(3 parts)}\\[4pt]
&= \vc{1}{0}+\vc{1}{0}+\vc{0}{1}+\vc{0}{1} && \text{(4 parts)}
\end{alignat*}
so $(2,2)$ has $9$ vector partitions in total.

\begin{notation}\label{notation-vec-par}
Let $|\alpha|=\sum_i\alpha_i$ denote the size of the vector $\alpha$. Let
$p(\alpha)$ denote the number of vector partitions of $\alpha$, and let
$p_n(\alpha)$ denote the number of vector partitions of $\alpha$ with at most
$n$ parts. Let $p_n^m(\alpha)$ denote the number of vector partitions of
$\alpha$ with at most $n$ parts, all of size $m$, and let
$p_n^{\leq m}(\alpha)$ denote the number of vector partitions of $\alpha$ with
at most $n$ parts, all of size at most $m$.
\end{notation}

\begin{notation}
For a partition $\lambda=(\lambda_1,\lambda_2,\dots,\lambda_l)$ and positive
integers $m,n$ with $mn-|\lambda|\geq\lambda_1$, let $\lambda[mn]$ denote the
padded partition $(mn-|\lambda|,\lambda_1,\lambda_2,\dots,\lambda_l)$. We write
$\ell(\lambda)$ for the length of $\lambda$, and we often regard a partition
$\lambda=(\lambda_1\geq \lambda_2\geq{\dots}\geq \lambda_{\ell(\lambda)})$ as the
vector $(\lambda_1,\lambda_2,\dots,\lambda_{\ell(\lambda)})$.
\end{notation}

\begin{notation}
If $\alpha$ is a vector with a negative component, then $p(\alpha)$,
$p_n(\alpha)$, $p_n^m(\alpha)$ and $p_n^{\le m}(\alpha)$ are defined to be zero.
\end{notation}

Monomial expansions of plethysms of complete symmetric functions are naturally
connected to vector-partition functions.

\begin{lemma}\label{monomial-complete} The monomial expansion of $h_n[h_m]$ is given by
    $$h_n[h_m]=\sum_{\lambda\vdash mn}p_n^m({\lambda})m_\lambda.$$
\end{lemma}
\begin{proof}
The monomial $m_\lambda$ can occur in the monomial expansion of $h_n[h_m]$ only
if $\ell(\lambda)\leq n$, so it is enough to evaluate $h_n[h_m]$ in $n$
variables. Write $h_n[h_m](x_1,x_2,\dots,x_n)=\sum_\beta c_\beta x^\beta$, where
$x^\beta=x_1^{\beta_1}x_2^{\beta_2}\cdots x_n^{\beta_n}$. Since
$h_m=\sum_{|\alpha|=m} x^\alpha$, where $\alpha$ runs over all compositions of
size $m$, the coefficient of $x^\lambda$ in the plethysm is the number of
multisets $\{\gamma^1,\gamma^2,\dots,\gamma^n\}$ of compositions with
$|\gamma^j|=m$ and $\lambda=\sum_{j=1}^n\gamma^j$. This is the number of vector
partitions of $\lambda$ into $n$ parts, each of size $m$.
\end{proof}

Using the Jacobi--Trudi identity \cite[p.~342]{stanley}, we expand $s_\lambda$
in terms of complete symmetric functions and obtain the following theorem.

\begin{theorem}\label{schur-exp-plethysm}
For all $\lambda$ with $\ell(\lambda)\leq n$,
$$\langle h_n[h_m],s_{\lambda}\rangle=\sum_{\sigma\in S_n}\sgn(\sigma)\, p_n^m(\lambda+\omega(\sigma)),$$
where $\omega(\sigma)(i)=\sigma(i)-i$. If $\ell(\lambda)>n$, then
$\langle h_n[h_m],s_{\lambda}\rangle=0$.
\end{theorem}
\begin{proof}
By the Jacobi--Trudi identity,
$$\langle h_n[h_m],s_\lambda\rangle=\sum_{\sigma\in S_n}\sgn(\sigma)\langle h_n[h_m],h_{\lambda+\omega(\sigma)}\rangle,$$
and $\langle h_n[h_m],h_{\lambda+\omega(\sigma)}\rangle$ is the coefficient of
$m_{\lambda+\omega(\sigma)}$ in the monomial expansion of $h_n[h_m]$. The
theorem follows from Lemma~\ref{monomial-complete}.
\end{proof}

\begin{lemma}\label{stability-of-vecpar}
Let $\mu=(\mu_1,\mu_2,\dots,\mu_l)$ be a vector (not necessarily a partition). Then
$$p_n^m(mn-|\mu|,\mu)=p_n^{\leq m}(\mu).$$
\end{lemma}
\begin{proof}
Suppose $(mn-|\mu|,\mu_1,\mu_2,\dots,\mu_l)=\sum_{j=1}^n\alpha^j$ with
$|\alpha^j|=m$ for all $j$, and let $\alpha_k^j$ denote the $k$-th component of
$\alpha^j$. Removing the first component $\alpha_1^j$ of each $\alpha^j$ gives a
vector partition of $\mu$, and since $|\alpha^j|=m$, each resulting vector has
size at most $m$. Conversely, if $\mu=\sum_{j=1}^n\beta^{j}$ (some of the
$\beta^j$ may be zero) with $|\beta^j|\leq m$, then setting
$\gamma^j=(m-|\beta^j|,\beta^j)$ gives a vector partition
$(mn-|\mu|,\mu)=\sum_{j=1}^n\gamma^j$ into $n$ parts of size $m$. These
constructions are mutually inverse.
\end{proof}

\begin{example}
Let $m=n=3$ and $\lambda=(2,2)$. A vector partition of $(5,2,2)$ into at most
$3$ parts, each of size $3$, is
$$
\vcc{5}{2}{2} = \vcc{3}{0}{0} + \vcc{1}{1}{1} + \vcc{1}{1}{1}.
$$
Removing the first component of each vector (and discarding the resulting zero
vector) gives a vector partition of $(2,2)$ into at most $3$ parts, each of
size at most $3$:
$$
\vc{2}{2} = \vc{1}{1} + \vc{1}{1}.
$$
This process can be reversed to recover the vector partition of $(5,2,2)$.
\end{example}

\begin{lemma}\label{no-role-of-m-and-n}
Let $\alpha$ be a vector. For all $m,n\geq |\alpha|$, we have $p_n^{\leq m}(\alpha)=p(\alpha)$.
\end{lemma}
\begin{proof}
A vector partition of $\alpha$ has at most $|\alpha|$ parts, and each of its
parts has size at most $|\alpha|$.
\end{proof}

The main result of this section is the following proposition, which
expresses a stability property of vector-partition functions.

\begin{proposition}\label{stability}
Let $\alpha=(\alpha_1,\alpha_2,\dots,\alpha_n)$ be a vector with $|\alpha|=mn$,
and let $d>m(n-1)-\alpha_1-1$. Then
$$p_n^{m+d}(\alpha_1+nd,\alpha_2,\dots,\alpha_n)=p_n^{m+d+1}(\alpha_1+n(d+1),\alpha_2,\dots,\alpha_n).$$
\end{proposition}
\begin{proof}
Given a vector partition of $(\alpha_1+nd,\alpha_2,\dots,\alpha_n)$ into $n$
parts of size $m+d$, adding $1$ to the first component of each part gives a
vector partition of $(\alpha_1+n(d+1),\alpha_2,\dots,\alpha_n)$ into $n$ parts
of size $m+d+1$. Conversely, let
$(\alpha_1+n(d+1),\alpha_2,\dots,\alpha_n)=\sum_{j=1}^n\gamma^j$ with
$|\gamma^j|=m+d+1$ for all $j$. We claim that $\gamma_1^j>0$ for all $j$.
Otherwise $\gamma_1^j=0$ for some $j$, and then
$m+d+1=|\gamma^j|\leq \alpha_2+\alpha_3+\dots+\alpha_n=mn-\alpha_1$, that is,
$d\le m(n-1)-\alpha_1-1$, a contradiction. Subtracting $1$ from the first
component of each part therefore gives a vector partition of
$(\alpha_1+nd,\alpha_2,\dots,\alpha_n)$ into $n$ parts of size $m+d$. The two
constructions are mutually inverse.
\end{proof}

\section{Stability of plethysm coefficients}\label{sec-stability}
We show that for a partition $\lambda$ and $m,n$ large enough, the plethysm
coefficient $\langle h_n[h_m],s_{\lambda[mn]}\rangle$ is independent of $m$ and
$n$; in particular,
$\langle h_n[h_m],s_{\lambda[mn]}\rangle=\langle h_m[h_n],s_{\lambda[mn]}\rangle$.
This gives a short proof of the main stability theorem of \cite{bow-pag}.

\begin{theorem}[Bowman--Paget stability]\label{proof1}
The double sequence $\{\langle h_n[h_m],s_{\lambda[mn]}\rangle\}_{m\geq1,n\geq1}$
is constant for $m,n\geq |\lambda|$.
\end{theorem}
\begin{proof}
Write $\ell=\ell(\lambda)$. By Theorem~\ref{schur-exp-plethysm},
\begin{align*}
\langle h_n[h_m],s_{\lambda[mn]}\rangle
&=\sum_{\sigma\in S_{\ell+1}}\sgn(\sigma)\, p_{n}^{m}\big(mn-|\lambda|+\sigma(1)-1,\lambda_1+\sigma(2)-2,\dots,\lambda_{\ell}+\sigma(\ell+1)-(\ell+1)\big)\\
&=\sum_{\sigma\in S_{\ell+1}}\sgn(\sigma)\, p_{n}^{\leq m}\big(\lambda_1+\sigma(2)-2,\dots,\lambda_{\ell}+\sigma(\ell+1)-(\ell+1)\big)\\
&=\sum_{\sigma\in S_{\ell+1}}\sgn(\sigma)\, p\big(\lambda_1+\sigma(2)-2,\dots,\lambda_{\ell}+\sigma(\ell+1)-(\ell+1)\big),
\end{align*}
which is independent of $m$ and $n$. The second equality follows from
Lemma~\ref{stability-of-vecpar}, since the entries of
$\lambda[mn]+\omega(\sigma)$ sum to $mn$ (where $\omega(\sigma)(i)=\sigma(i)-i$); the third follows from
Lemma~\ref{no-role-of-m-and-n}, since the vectors involved have size at most
$|\lambda|\le m,n$.
\end{proof}

\begin{notation}
Let $\lambda+(k)$ denote the partition $(\lambda_1+k,\lambda_2,\lambda_3,\dots)$.
\end{notation}

Next, we give a short proof of a stability result of Brion \cite{brion}: the
limit $\lim_{d\to\infty}\langle h_n[h_{m+d}],s_{\lambda+(nd)}\rangle$ exists.
It is enough to show that
$\langle h_n[h_{m+d}],s_{\lambda+(nd)}\rangle=\langle h_n[h_{m+d+1}],s_{\lambda+(n(d+1))}\rangle$
for $d$ large enough. Our argument is in the same spirit as Colmenarejo's
combinatorial proof \cite{colmenarejo-stability}, which also passes through the
monomial expansion of $h_n[h_m]$; working directly with vector partitions makes
the bijection particularly short.

\begin{theorem}[Brion's stability]\label{proof2}
Let $\lambda=(\lambda_1\ge\lambda_2\ge\dots\ge\lambda_n)\vdash mn$. If
$d>mn-m-\lambda_1-1$, then
$$\langle h_n[h_{m+d}],s_{\lambda+(nd)}\rangle=\langle h_n[h_{m+d+1}],s_{\lambda+(n(d+1))}\rangle.$$
\end{theorem}
\begin{proof}
By Theorem~\ref{schur-exp-plethysm},
\begin{align*}
&\langle h_n[h_{m+d+1}],s_{\lambda+(n(d+1))}\rangle\\
&\quad=\sum_{\sigma\in S_n}\sgn(\sigma)\, p_{n}^{m+d+1}\big(\lambda_1+n(d+1)+\sigma(1)-1,\lambda_2+\sigma(2)-2,\dots,\lambda_n+\sigma(n)-n\big)\\
&\quad=\sum_{\sigma\in S_n}\sgn(\sigma)\, p_{n}^{m+d}\big(\lambda_1+nd+\sigma(1)-1,\lambda_2+\sigma(2)-2,\dots,\lambda_n+\sigma(n)-n\big)\\
&\quad=\langle h_n[h_{m+d}],s_{\lambda+(nd)}\rangle,
\end{align*}
where the second equality follows from Proposition~\ref{stability} applied to
$\alpha=\lambda+\omega(\sigma)$ (where $\omega(\sigma)(i)=\sigma(i)-i$), whose first component is
$\lambda_1+\sigma(1)-1\ge\lambda_1$.
\end{proof}

\section{Stable plethysm coefficients and the plethysm with \texorpdfstring{$H-h_1$}{H-h1}}\label{stable-plethysm}
In this section, we relate the stable plethysm coefficients to plethysms with
$H-h_1$, in analogy with a formula for restriction coefficients.

Let $f$ be a function on $\Z^n$. For $1\le i\le n$ and $k\in\Z$, define the
operator $E_i^k$ by
\begin{equation}\label{def-of-E-i-k}
(E_i^k\bullet f)(a_1,\dots,a_i,\dots,a_n)=f(a_1,\dots,a_i+k,\dots,a_n),
\end{equation}
and write $E_i=E_i^1$. The operators $E_1,\dots,E_n$ commute. By Littlewood's
plethystic formula (see Section~\ref{section-on-polynomial-rep}), the
restriction coefficients are given by
$r_\lambda^\mu=\langle s_\lambda,s_\mu[H]\rangle$, and by
\cite[Corollary~4.2]{narayanan},
\begin{equation*}
    r_\lambda^{(n)}=\sum_{\sigma\in S_n}\sgn(\sigma)\,p_n(\lambda+\omega(\sigma)),
\end{equation*}
where $\omega(\sigma)(i)=\sigma(i)-i$. In terms of the operators $E_i$, this
formula reads
\begin{equation*}
    r_\lambda^{(n)}=\sum_{\sigma\in S_n}\sgn(\sigma)\bigg(\Big(\prod_{i=1}^nE_i^{\sigma(i)-i}\Big)\bullet p_n\bigg)(\lambda)=\Big(\det\big[E_i^{j-i}\big]\bullet p_n\Big)(\lambda),
\end{equation*}
where $[E_i^{j-i}]$ denotes the $n\times n$ matrix whose $(i,j)$-th entry is $E_i^{j-i}$.

\begin{example}
We have $r_\lambda^{(2)}=p_2(\lambda_1,\lambda_2)-p_2(\lambda_1+1,\lambda_2-1)$,
which can also be written as
\begin{equation*}
    \big((1-E_1E_2^{-1})\bullet p_2\big)(\lambda)=\Bigg(\det\begin{pmatrix}
        1&E_1\\E_2^{-1}&1
    \end{pmatrix}\bullet p_2\Bigg) (\lambda).
\end{equation*}
\end{example}

\begin{notation}\label{notation-p-geq-2}
Let $p^{\geq 2}(\alpha)$ denote the number of vector partitions of $\alpha$
all of whose parts have size at least $2$, and let $p_n^{\geq 2}(\alpha)$
denote the number of such vector partitions with at most $n$ parts.
\end{notation}

\begin{lemma}\label{partition-without-e-i}
For every vector $\lambda=(\lambda_1,\dots,\lambda_n)$,
$$p^{\geq 2}(\lambda)=\bigg(\Big(\prod_{i=1}^n(1-E_i^{-1})\Big)\bullet p\bigg)(\lambda).$$
\end{lemma}
\begin{proof} We write $[n]$ for the set $\{1,2,\dots,n\}$.
Let $e_i$ denote the $i$-th coordinate vector. A vector partition of $\lambda$
has all parts of size at least $2$ exactly when none of its parts is a
coordinate vector $e_i$. For $A\subseteq[n]$, removing one copy of $e_i$ for
each $i\in A$ is a bijection from the vector partitions of $\lambda$ having
every $e_i$ ($i\in A$) as a part to the vector partitions of
$\lambda-\sum_{i\in A}e_i$. Hence, by inclusion--exclusion,
\begin{align*}
    p^{\geq 2}(\lambda)=\sum_{A\subseteq[n]}(-1)^{|A|}\,p\Big(\lambda-\sum_{i\in A}e_i\Big)
    =\bigg(\Big(\prod_{i=1}^n(1-E_i^{-1})\Big)\bullet p\bigg)(\lambda).
\end{align*}
\end{proof}

The next lemma is crucial for what follows.

\begin{lemma}\label{determinantal-identity}
We have the equality of operators
\begin{equation*}
    \prod_{i=1}^n(1-E_i^{-1})\cdot\Big(\sum_{\sigma\in S_n}\sgn(\sigma)\prod_{i=1}^nE_i^{\sigma(i)-i}\Big)=\sum_{\sigma\in S_{n+1}}\sgn(\sigma)\prod_{i=1}^nE_i^{\sigma(i+1)-(i+1)}.
\end{equation*}
\end{lemma}
\begin{proof}
For variables $x_0,x_1,\dots,x_n$, let
$$V(x_0,\dots,x_n)=\det\big[x_i^{j}\big]_{0\le i,j\le n}=\prod_{0\le i<j\le n}(x_j-x_i)$$
denote the Vandermonde determinant, and define $V(x_1,\dots,x_n)$ similarly.
Since the operators $E_1,\dots,E_n$ commute, we may substitute them for the
variables in any polynomial identity. Multiplying both sides of the claimed
identity by $\prod_{i=1}^nE_i^{i+1}$, it is enough to show that
$$\prod_{i=1}^n(E_i-1)\cdot\Big(\sum_{\sigma\in S_n}\sgn(\sigma)\prod_{i=1}^nE_i^{\sigma(i)}\Big)=\sum_{\sigma\in S_{n+1}}\sgn(\sigma)\prod_{i=1}^nE_i^{\sigma(i+1)}.$$
Factoring $E_i$ out of the $i$-th row of the determinant, the sum on the left
equals
$$\det\big[E_i^j\big]_{1\le i,j\le n}=(E_1E_2\cdots E_n)\cdot V(E_1,\dots,E_n).$$
Setting $E_0:=1$, the right-hand side equals
\begin{align*}
\det\big[E_i^{j}\big]_{0\le i\le n,\ 1\le j\le n+1}
&=(E_1E_2\cdots E_n)\cdot V(1,E_1,\dots,E_n)\\
&=(E_1E_2\cdots E_n)\cdot\prod_{i=1}^n(E_i-1)\cdot V(E_1,\dots,E_n),
\end{align*}
since the factors of $V(1,E_1,\dots,E_n)$ involving $E_0=1$ are exactly
$E_i-1$ for $1\le i\le n$. This proves the identity.
\end{proof}

Applying this lemma, we obtain the following.

\begin{theorem}\label{equality-of-vector-partitions}
For every vector $\lambda=(\lambda_1,\lambda_2,\dots,\lambda_n)$,
\begin{align*}
&\sum_{\sigma\in S_{n+1}}\sgn(\sigma)\,p\big(\lambda_1+\sigma(2)-2,\lambda_2+\sigma(3)-3,\dots,\lambda_n+\sigma(n+1)-(n+1)\big)\\
&\qquad=\sum_{\sigma\in S_n}\sgn(\sigma)\,p^{\geq 2}\big(\lambda_1+\sigma(1)-1,\lambda_2+\sigma(2)-2,\dots,\lambda_n+\sigma(n)-n\big).
\end{align*}
\end{theorem}
\begin{proof}
Since the operators $E_i$ commute, Lemma~\ref{determinantal-identity}
and Lemma~\ref{partition-without-e-i} give
\begin{align*}
&\sum_{\sigma\in S_{n+1}}\sgn(\sigma)\,p\big(\lambda_1+\sigma(2)-2,\dots,\lambda_n+\sigma(n+1)-(n+1)\big)\\
&\qquad=\bigg(\Big(\sum_{\sigma\in S_{n+1}}\sgn(\sigma)\prod_{i=1}^nE_i^{\sigma(i+1)-(i+1)}\Big)\bullet p\bigg)(\lambda)\\
&\qquad=\bigg(\Big(\sum_{\sigma\in S_n}\sgn(\sigma)\prod_{i=1}^nE_i^{\sigma(i)-i}\Big)\bullet\Big(\Big(\prod_{i=1}^n(1-E_i^{-1})\Big)\bullet p\Big)\bigg)(\lambda) &&\text{(Lemma~\ref{determinantal-identity})}\\
&\qquad=\bigg(\Big(\sum_{\sigma\in S_n}\sgn(\sigma)\prod_{i=1}^nE_i^{\sigma(i)-i}\Big)\bullet p^{\ge2}\bigg)(\lambda) &&\text{(Lemma~\ref{partition-without-e-i})}\\
&\qquad=\sum_{\sigma\in S_n}\sgn(\sigma)\,p^{\geq 2}\big(\lambda_1+\sigma(1)-1,\dots,\lambda_n+\sigma(n)-n\big).
\end{align*}
\end{proof}

The next lemma relates the plethysm coefficients $\langle h_n[H-h_1],s_\lambda\rangle$
to the vector-partition function $p^{\ge 2}(\cdot)$.

\begin{lemma}\label{greater-2-lemma}
For every partition $\lambda$,
$$\langle h_{n}[h_0+h_2+h_3+h_4+\cdots],s_\lambda\rangle=\sum_{\sigma\in S_{\ell(\lambda)}}\sgn(\sigma)\,p_{n}^{\geq 2}\big(\lambda_1+\sigma(1)-1,\lambda_2+\sigma(2)-2,\dots,\lambda_{\ell(\lambda)}+\sigma(\ell(\lambda))-\ell(\lambda)\big).$$
\end{lemma}
\begin{proof}
As in the proof of Lemma~\ref{monomial-complete}, the coefficient of $m_\mu$ in
the monomial expansion of $h_{n}[h_0+h_2+h_3+h_4+\cdots]$ is the number of
vector partitions of $\mu$ with at most $n$ parts, each of size at least $2$
(the summand $h_0$ accounts for the empty parts). The lemma follows from the
Jacobi--Trudi identity, as in Theorem~\ref{schur-exp-plethysm}.
\end{proof}

We now state the main result of this section.

\begin{theorem}\label{omitting-h-1}
For all $m,n\geq |\lambda|$,
$$\langle h_n[h_m],s_{\lambda[mn]}\rangle=\langle h_{\lfloor|\lambda|/2\rfloor}[H-h_1],s_\lambda\rangle,$$
where $H=\sum_{i\ge0}h_i$.
\end{theorem}
\begin{proof}
Write $\ell=\ell(\lambda)$. For every $\sigma\in S_{\ell}$, the vector
$(\lambda_1+\sigma(1)-1,\dots,\lambda_{\ell}+\sigma(\ell)-\ell)$ has size
$|\lambda|$, so any of its vector partitions with all parts of size at least
$2$ has at most $\lfloor|\lambda|/2\rfloor$ parts. By the proof of
Theorem~\ref{proof1}, for all $m,n\geq |\lambda|$,
\begin{align*}
\langle h_n[h_m],s_{\lambda[mn]}\rangle
&=\sum_{\sigma\in S_{\ell+1}}\sgn(\sigma)\, p\big(\lambda_1+\sigma(2)-2,\dots,\lambda_{\ell}+\sigma(\ell+1)-(\ell+1)\big)\\
&=\sum_{\sigma\in S_{\ell}}\sgn(\sigma)\,p^{\geq 2}\big(\lambda_1+\sigma(1)-1,\dots,\lambda_{\ell}+\sigma(\ell)-\ell\big) &&\text{(Theorem~\ref{equality-of-vector-partitions})}\\
&=\sum_{\sigma\in S_{\ell}}\sgn(\sigma)\,p_{\lfloor |\lambda|/2\rfloor}^{\geq 2}\big(\lambda_1+\sigma(1)-1,\dots,\lambda_{\ell}+\sigma(\ell)-\ell\big)\\
&=\langle h_{\lfloor |\lambda|/2\rfloor}[h_0+h_2+h_3+h_4+\cdots],s_\lambda\rangle &&\text{(Lemma~\ref{greater-2-lemma})}\\
&=\langle h_{\lfloor |\lambda|/2\rfloor}[H-h_1],s_\lambda\rangle.
\end{align*}
\end{proof}

Tables~\ref{table-22}--\ref{table-6} show the coefficients
$\langle h_n[h_m], s_{\lambda[mn]}\rangle$ for small $m$ and $n$.

\begin{remark}
The bound $m,n\ge|\lambda|$ is sharp for one-row partitions: for $\lambda=(6)$,
the stable value is reached only when $m,n\geq 6$ (Table~\ref{table-6}). For
other partitions the coefficients stabilize earlier, as explained by
Theorem~\ref{thm-perp}: for $\lambda=(2,2)$ they are constant for $m\ge2$ and
$n\ge3$, and for $\lambda=(4,2)$ they are constant for $m\ge4$ and $n\ge5$.
The latter range cannot be improved in either $m$ or $n$ (Table~\ref{table-42}).
\end{remark}

\begin{table}[H]\centering
\caption{$\langle h_n[h_m], s_{\lambda[mn]}\rangle$ for $\lambda=(2,2)$, $|\lambda|=4$. Shaded cells: $m,n\ge|\lambda|$. A dash means $\lambda[mn]$ is not a partition.}\label{table-22}
\begin{tabular}{c|ccccccc}
$m \backslash n$ & 1 & 2 & 3 & 4 & 5 & 6 & 7 \\ \hline
1 & -- & -- & -- & -- & -- & 0 & 0 \\
2 & -- & -- & 1 & 1 & 1 & 1 & 1 \\
3 & -- & 0 & 1 & 1 & 1 & 1 & 1 \\
4 & -- & 0 & 1 & \cellcolor{gray!25}1 & \cellcolor{gray!25}1 & \cellcolor{gray!25}1 & \cellcolor{gray!25}1 \\
5 & -- & 0 & 1 & \cellcolor{gray!25}1 & \cellcolor{gray!25}1 & \cellcolor{gray!25}1 & \cellcolor{gray!25}1 \\
6 & 0 & 0 & 1 & \cellcolor{gray!25}1 & \cellcolor{gray!25}1 & \cellcolor{gray!25}1 & \cellcolor{gray!25}1 \\
7 & 0 & 0 & 1 & \cellcolor{gray!25}1 & \cellcolor{gray!25}1 & \cellcolor{gray!25}1 & \cellcolor{gray!25}1 \\
\end{tabular}
\end{table}

\begin{table}[H]\centering
\caption{$\langle h_n[h_m], s_{\lambda[mn]}\rangle$ for $\lambda=(4,2)$, $|\lambda|=6$. Shaded cells: $m,n\ge|\lambda|$. A dash means $\lambda[mn]$ is not a partition.}\label{table-42}
\begin{tabular}{c|ccccccccc}
$m \backslash n$ & 1 & 2 & 3 & 4 & 5 & 6 & 7 & 8 & 9 \\ \hline
1 & -- & -- & -- & -- & -- & -- & -- & -- & -- \\
2 & -- & -- & -- & -- & 1 & 1 & 1 & 1 & 1 \\
3 & -- & -- & -- & 1 & 2 & 2 & 2 & 2 & 2 \\
4 & -- & -- & 1 & 2 & 3 & 3 & 3 & 3 & 3 \\
5 & -- & 0 & 1 & 2 & 3 & 3 & 3 & 3 & 3 \\
6 & -- & 0 & 1 & 2 & 3 & \cellcolor{gray!25}3 & \cellcolor{gray!25}3 & \cellcolor{gray!25}3 & \cellcolor{gray!25}3 \\
7 & -- & 0 & 1 & 2 & 3 & \cellcolor{gray!25}3 & \cellcolor{gray!25}3 & \cellcolor{gray!25}3 & \cellcolor{gray!25}3 \\
8 & -- & 0 & 1 & 2 & 3 & \cellcolor{gray!25}3 & \cellcolor{gray!25}3 & \cellcolor{gray!25}3 & \cellcolor{gray!25}3 \\
9 & -- & 0 & 1 & 2 & 3 & \cellcolor{gray!25}3 & \cellcolor{gray!25}3 & \cellcolor{gray!25}3 & \cellcolor{gray!25}3 \\
\end{tabular}
\end{table}

\begin{table}[H]\centering
\caption{$\langle h_n[h_m], s_{\lambda[mn]}\rangle$ for $\lambda=(6)$, $|\lambda|=6$. Shaded cells: $m,n\ge|\lambda|$. A dash means $\lambda[mn]$ is not a partition.}\label{table-6}
\begin{tabular}{c|ccccccccc}
$m \backslash n$ & 1 & 2 & 3 & 4 & 5 & 6 & 7 & 8 & 9 \\ \hline
1 & -- & -- & -- & -- & -- & -- & -- & -- & -- \\
2 & -- & -- & -- & -- & -- & 1 & 1 & 1 & 1 \\
3 & -- & -- & -- & 1 & 1 & 2 & 2 & 2 & 2 \\
4 & -- & -- & 1 & 2 & 2 & 3 & 3 & 3 & 3 \\
5 & -- & -- & 1 & 2 & 2 & 3 & 3 & 3 & 3 \\
6 & -- & 1 & 2 & 3 & 3 & \cellcolor{gray!25}4 & \cellcolor{gray!25}4 & \cellcolor{gray!25}4 & \cellcolor{gray!25}4 \\
7 & -- & 1 & 2 & 3 & 3 & \cellcolor{gray!25}4 & \cellcolor{gray!25}4 & \cellcolor{gray!25}4 & \cellcolor{gray!25}4 \\
8 & -- & 1 & 2 & 3 & 3 & \cellcolor{gray!25}4 & \cellcolor{gray!25}4 & \cellcolor{gray!25}4 & \cellcolor{gray!25}4 \\
9 & -- & 1 & 2 & 3 & 3 & \cellcolor{gray!25}4 & \cellcolor{gray!25}4 & \cellcolor{gray!25}4 & \cellcolor{gray!25}4 \\
\end{tabular}
\end{table}

\section{Polynomial representations of \texorpdfstring{$GL_n(\C)$}{GLn(C)} and restriction coefficients}\label{section-on-polynomial-rep}
A group homomorphism $\phi\colon GL_n(\C)\to GL(V)$ is called a
\emph{polynomial representation} of $GL_n(\C)$ if the matrix coefficients of
$\phi(g)$ with respect to a basis of $V$ are polynomial functions in the
entries of $g$. Irreducible polynomial representations of $GL_n(\C)$ of degree
$d$ are parametrized by partitions of $d$ with at most $n$ parts; these are the
\emph{Weyl modules}, denoted $W_\lambda(\C^n)$. For a symmetric function $f$
(possibly an infinite sum of homogeneous terms, such as $H$), we write $(f)_d$
for its homogeneous component of degree $d$.

The symmetric group $S_n\subseteq GL_n(\C)$ is the subgroup of permutation
matrices. Finding a satisfactory combinatorial interpretation for the
multiplicity $r_\lambda^\mu$ of the Specht module $V_\mu$ in the restriction
$W_\lambda(\C^n)\big\downarrow_{S_n}$ remains an open problem at the
intersection of representation theory and algebraic combinatorics. By a
well-known formula of Littlewood \cite{littlewood}, these \emph{restriction
coefficients} are given by
$$r_\lambda^\mu=\langle s_\lambda,s_\mu[H]\rangle,\qquad\text{where } H=\sum_{i\geq 0}h_i.$$
In particular, $r_\lambda^{(n)}=\langle s_\lambda,h_n[H]\rangle$.
Narayanan, Paul, Prasad and Srivastava \cite{narayanan} gave a conceptual
explanation of this formula. They defined a functor $\mathrm{Ind}^d$ from the
category $\mathrm{Rep}\,S_n$ to the category $\mathrm{Rep}^d\,GL_n(\C)$ of
degree $d$ polynomial representations of $GL_n(\C)$, right adjoint to
restriction, in such a way that $(s_\mu[H])_d$ is the character of the induced
module $\mathrm{Ind}^d(V_\mu)$.

\subsection{Polynomial induction}
Let $X=(x_{ij})_{1\le i,j\le n}$ denote a generic element of $GL_n(\C)$, so
that the entries $x_{ij}$ are the coordinate functions on $GL_n(\C)$. Any
polynomial function on $GL_n(\C)$ can be written as $f(X)=\sum_{A}c_AX^A$,
where $X^A=\prod_{1\leq i,j\leq n}x_{ij}^{a_{ij}}$ is indexed by the matrix
$A=(a_{ij})_{1\leq i,j\leq n}$ of non-negative integers. We identify $S_n$
with the group of permutation matrices $w\in GL_n(\C)$, so that $wX$ is
obtained from $X$ by permuting its rows. Consider the space of polynomial
functions of degree $d$ that are invariant under row permutations, i.e.,
$f(wX)=f(X)$ for all $w\in S_n$. The group $GL_n(\C)$ acts on this space by
right multiplication, $(g\cdot f)(X)=f(Xg)$ for $g\in GL_n(\C)$, where $Xg$
is the matrix product. By \cite[Theorem~4.7]{narayanan}, the character of this
representation is
$$\mathrm{Ch}\big(\mathrm{Ind}^d\ _{S_n}^{GL_n(\C)}\mathbf{1}\big)=(h_n[H])_d.$$
We now modify this induced module and relate its character to stable Foulkes
coefficients.

\section{Modified polynomial induction}\label{modified-poly-rep}
We keep the notation of Section~\ref{section-on-polynomial-rep}: $X=(x_{ij})$
is a generic element of $GL_n(\C)$, a permutation matrix $w\in S_n$ acts on
polynomial functions by $f(X)\mapsto f(wX)$, permuting the rows of $X$, and
$g\in GL_n(\C)$ acts by $(g\cdot f)(X)=f(Xg)$. For a matrix $A=(a_{ij})$, let
$\mathrm{row}(A)=\big\{\sum_{j}a_{1j},\sum_{j}a_{2j},\dots, \sum_{j}a_{nj}\big\}$
be its set of row sums. We consider the polynomial functions of degree $d$
that are invariant under permutation of rows and involve only monomials $X^A$
with \emph{no} row sum equal to $1$:
\begin{equation}\label{def-modified-poly-ind}
\mathrm{Ind}_{\hat{1}}^d\ _{S_n}^{GL_n(\C)}\mathbf{1}:=\Big\{f(X)=\sum_Ac_AX^A\ \Big|\ f(wX)=f(X)\ \text{for all}\ w\in S_n,\ \text{and}\ c_A\neq0\Rightarrow 1\notin \mathrm{row}(A)\Big\}.
\end{equation}

\begin{proposition}
$\mathrm{Ind}_{\hat{1}}^d\ _{S_n}^{GL_n(\C)}\mathbf{1}$ is a
$GL_n(\C)$-subrepresentation of $\mathrm{Ind}^d\ _{S_n}^{GL_n(\C)}\mathbf{1}$.
\end{proposition}
\begin{proof}
Let $g=(g_{kj})\in GL_n(\C)$. Invariance under row permutations is preserved
by the action, since
$(g\cdot f)(wX)=f(wXg)=f(Xg)$. It remains to check that the row-sum condition
is preserved. Let $R_i(A) = \sum_{j=1}^n a_{ij}$ denote the $i$-th row sum of
$A$. We have
\begin{equation*}
    (Xg)^A = \prod_{i=1}^n \prod_{j=1}^n (Xg)_{ij}^{a_{ij}} = \prod_{i=1}^n \prod_{j=1}^n \Big( \sum_{k=1}^n x_{ik} g_{kj} \Big)^{a_{ij}}.
\end{equation*}
For a fixed $i$, each $(Xg)_{ij} = \sum_{k} x_{ik} g_{kj}$ is homogeneous of
degree $1$ in the variables $x_{i1}, \dots, x_{in}$ of the $i$-th row of $X$,
so $P_i(X) = \prod_{j=1}^n (Xg)_{ij}^{a_{ij}}$ is homogeneous of degree
$R_i(A)$ in these variables and involves no other variables. Expanding,
\begin{equation*}
    (Xg)^A=\prod_{i=1}^n P_i(X) = \sum_{B} K_{A,B}(g)\, X^B,
\end{equation*}
and every monomial $X^B$ with $K_{A,B}(g)\neq0$ has degree $R_i(A)$ in the
variables of the $i$-th row, that is, $R_i(B)=R_i(A)$ for all $i$. Hence
$\mathrm{row}(B)=\mathrm{row}(A)$, and $1\notin \mathrm{row}(A)$ implies
$1\notin \mathrm{row}(B)$.
\end{proof}

\begin{proposition}\label{Equal-inner-products}
$\mathrm{Ch}\big(\mathrm{Ind}_{\hat{1}}^d\ _{S_n}^{GL_n(\C)}\mathbf{1}\big)=(h_n[H-h_1])_d=(h_n[h_0+h_2+h_3+\cdots])_d.$
\end{proposition}
\begin{proof}
For $1\le i\le n$, let $x_i=(x_{i1},\dots,x_{in})$ denote the $i$-th row of
$X$, regarded as a row vector. Since the $i$-th row of $Xg$ is $x_ig$, the
action $(g\cdot f)(X)=f(Xg)$ preserves the space of polynomials in $x_i$ alone,
acting there by $f(x_i)\mapsto f(x_ig)$. Let $W_i$ be the subspace of
polynomials in $x_i$ with no terms of degree $1$. Taking homogeneous
components, $W_i\cong\bigoplus_{k\ne1}\mathrm{Sym}^k(\C^n)$ as
$GL_n(\C)$-modules, so
\begin{equation*}
    \mathrm{Ch}(W_i) = \sum_{k\neq 1} h_k = h_0 + h_2 + h_3 + h_4 + \cdots.
\end{equation*}
A polynomial in all the variables $x_{ij}$ has no monomial with a row sum
equal to $1$ exactly when it lies in $W_1\otimes W_2\otimes\cdots\otimes W_n$.
The action of $S_n$ by permuting the rows of $X$ permutes these tensor
factors, and all the $W_i$ are isomorphic to $W:=W_1$, so the subspace of
$S_n$-invariants is
\begin{equation*}
    \Big( \bigotimes_{i=1}^n W_i \Big)^{S_n} \cong \mathrm{Sym}^n(W),
\end{equation*}
whose character is $h_n[\mathrm{Ch}(W)]=h_n[h_0 + h_2 + h_3 + \cdots]$. Taking
the component of degree $d$ gives the result.
\end{proof}

\begin{theorem}\label{final-corollary}
For all $m,n\geq |\lambda|$ and $N\ge\ell(\lambda)+1$,
$$\langle W_{\lambda[mn]},\mathrm{Sym}^n(\mathrm{Sym}^m(\C^N))\rangle=\langle W_\lambda(\C^{\lfloor|\lambda|/2\rfloor}),\mathrm{Ind}_{\hat{1}}^{|\lambda|}\ _{S_{\lfloor|\lambda|/2\rfloor}}^{GL_{\lfloor|\lambda|/2\rfloor}}\mathbf{1}\rangle.$$
\end{theorem}
\begin{proof}
Since $N\ge\ell(\lambda[mn])$ and the character of
$\mathrm{Sym}^n(\mathrm{Sym}^m(\C^N))$ is $h_n[h_m](x_1,\dots,x_N)$, the
left-hand side equals $\langle h_n[h_m],s_{\lambda[mn]}\rangle$. By
Proposition~\ref{Equal-inner-products},
$$\mathrm{Ch}\big(\mathrm{Ind}_{\hat{1}}^{|\lambda|}\ _{S_{\lfloor |\lambda|/2\rfloor}}^{GL_{\lfloor |\lambda|/2\rfloor}(\C)}\mathbf{1}\big)=(h_{\lfloor |\lambda|/2\rfloor}[H-h_1])_{|\lambda|},$$
evaluated in $\lfloor |\lambda|/2\rfloor$ variables. The corollary follows from
Theorem~\ref{omitting-h-1}.
\end{proof}

Thus, for $m,n\ge|\lambda|$, the multiplicity of $W_{\lambda[mn]}$ in
$\mathrm{Sym}^n(\mathrm{Sym}^m(\C^N))$ equals the multiplicity of
$W_\lambda(\C^{\lfloor |\lambda|/2\rfloor})$ in
$\mathrm{Ind}_{\hat{1}}^{|\lambda|}\ _{S_{\lfloor |\lambda|/2\rfloor}}^{GL_{\lfloor |\lambda|/2\rfloor}}\mathbf{1}$,
which is independent of $m$ and $n$. In particular, the stable Foulkes
coefficient of $\lambda$ is zero when $\ell(\lambda)>\lfloor|\lambda|/2\rfloor$.

\section{Stability using the action of \texorpdfstring{$h_r^\perp$}{hr-perp}}\label{stability-using-perp}

We recall perp operators on the ring of symmetric functions; for details, see
\cite[Ch.~I, Sec.~5]{macdonald} and \cite{zabrocki}. For applications of perp
operators to the computation of plethysm coefficients, see \cite{colmenarejo}.

\begin{definition}
For a symmetric function $f$, let $f^\perp$ be the adjoint of multiplication by
$f$ with respect to the Hall inner product:
$\langle f^\perp g,\, u\rangle=\langle g,\, f u\rangle$ for all $g,u$.
\end{definition}

We use the following standard facts:
\begin{equation}\label{e-perp}
e_i^\perp s_\lambda=s_{\lambda/(1^i)},
\end{equation}
\begin{equation}\label{sum-with-perp}
\sum_{r\ge0} (h_r^\perp g)[X]=g[X+1],
\end{equation}
\begin{equation}\label{co-prod-for-h-n}
    h_n[X+Y]=\sum_{a+b=n}h_a[X]\,h_b[Y].
\end{equation}
\begin{equation}
    H[X+Y]=H[X]\cdot H[Y]
\end{equation}

We also set $A_m:=h_1+h_2+\cdots+h_m$.

\begin{lemma}\label{lam-1-ge-m}
If $\langle h_n[h_m],s_\lambda\rangle>0$, then $\lambda_1\ge m$.
\end{lemma}
\begin{proof}
Since $\mathrm{Sym}^n(\mathrm{Sym}^m(\C^N))\subseteq \bigotimes^n{\mathrm{Sym}^m}(\C^N)$
for all $N$, the difference $h_m^n-h_n[h_m]$ is Schur-positive. Moreover,
$h_m^n=\sum_{\mu\trianglerighteq (m^n)}K_{\mu,(m^n)}s_\mu$, where
$K_{\mu,(m^n)}$ is the Kostka number, that is, the number of semistandard
Young tableaux of shape $\mu$ and content $(m^n)$. Hence
$\langle h_n[h_m],s_\lambda\rangle>0$ implies
$\langle h_m^n,s_\lambda\rangle>0$, so
$\lambda\trianglerighteq (m^n)$ and in particular $\lambda_1\ge m$.
\end{proof}

Next, we compute the action of $h_r^\perp$ on plethysms of complete symmetric
functions.

\begin{lemma}\label{lem:hperp}
$h_r^\perp\, h_n[h_m]$ is the degree $mn-r$ part of $\sum_{j=0}^{n}h_j[A_m]$.
\end{lemma}
\begin{proof}
Since $h_m[X+1]=\sum_{i=0}^m h_i[X]$, we have $h_m[X+1]=1+A_m$ as a symmetric
function in $X$. Using also $h_i[1]=1$ for all $i$,
\begin{align*}
\sum_{r\ge0} h_r^\perp\, h_n[h_m]&=h_n[1+A_m] &&\text{(by \eqref{sum-with-perp})}\\
&=\sum_{i+j=n}h_i[1]\,h_j[A_m] &&\text{(by \eqref{co-prod-for-h-n})}\\
&=\sum_{j=0}^nh_j[A_m]
\end{align*}
Since $h_r^\perp h_n[h_m]$ is homogeneous of degree $mn-r$, comparing
components of each degree proves the lemma.
\end{proof}

\medskip
\noindent\textbf{Removing the first row.} Assume that
$mn-|\lambda|\ge \lambda_1$, so that $\lambda[mn]$ is a partition. Expanding
the Jacobi--Trudi determinant
$s_{\lambda[mn]}=\det\big(h_{(\lambda[mn])_i-i+j}\big)_{1\leq i,j\leq \ell(\lambda)+1}$
along its first row gives
$$s_{\lambda[mn]}=\sum_{i\ge0}(-1)^i\, h_{mn-|\lambda|+i}\; e_i^\perp s_\lambda .$$
Taking the inner product with $h_n[h_m]$ and applying Lemma~\ref{lem:hperp}
(only the component of degree $|\lambda|-i$ pairs non-trivially with
$e_i^\perp s_\lambda$),
\begin{align*}\langle h_n[h_m],s_{\lambda[mn]}\rangle
&=\sum_{i\ge0}(-1)^i\Big\langle h_{mn-|\lambda|+i}^\perp h_n[h_m],\ e_i^\perp s_\lambda\Big\rangle
=\sum_{i\ge0}(-1)^i\Big\langle \sum_{j=0}^{n}h_j[A_m],\ e_i^\perp s_\lambda\Big\rangle\\
&=\sum_{i\ge0}(-1)^i\Big\langle e_i\sum_{j=0}^{n}h_j[A_m],\ s_\lambda\Big\rangle
=\Big\langle \frac{1}{H}\sum_{j=0}^{n}h_j[A_m],\ s_\lambda\Big\rangle,
\end{align*}
where $\frac{1}{H}=\sum_{i\ge0}(-1)^ie_i$.

\begin{theorem}\label{thm-perp}
For $m\ge\lambda_1$ and $n\ge\lfloor(|\lambda|+\lambda_1)/2\rfloor$
(in particular, for $m,n\ge|\lambda|$), whenever $\lambda[mn]$ is a partition,
$$\langle h_n[h_m],s_{\lambda[mn]}\rangle=\langle H[H-1-h_1], s_\lambda\rangle.$$
\end{theorem}
\begin{proof}
Since $e_i^\perp s_\lambda=s_{\lambda/(1^i)}$ involves only Schur functions
$s_\nu$ with $\nu\subseteq\lambda$, the pairing
$\big\langle \frac1H\, g,\, s_\lambda\big\rangle=\sum_i(-1)^i\langle g, e_i^\perp s_\lambda\rangle$
vanishes whenever every Schur function $s_\nu$ in the expansion of $g$ has
$\nu_1>\lambda_1$ or $|\nu|>|\lambda|$. We show that once $m\ge\lambda_1$ and
$n\ge\lfloor(|\lambda|+\lambda_1)/2\rfloor$, increasing $m$ or $n$ does not
change the inner product computed above.

\begin{itemize}
\item \emph{Changing $m$ to $m+1$.} By \eqref{co-prod-for-h-n},
\begin{align*}
\frac{1}{H}\sum_{j=0}^{n}h_j[A_{m+1}]- \frac{1}{H}\sum_{j=0}^{n}h_j[A_m]
&=\frac{1}{H}\sum_{j=0}^n\sum_{t=0}^{j} h_{j-t}[A_m]\, h_t[h_{m+1}]-\frac{1}{H}\sum_{j=0}^{n}h_j[A_m]\\
&=\frac{1}{H}\sum_{j=0}^{n}\sum_{t=1}^{j} h_{j-t}[A_m]\, h_t[h_{m+1}],
\end{align*}
since the terms with $t=0$ cancel. Each remaining term contains a factor
$h_t[h_{m+1}]$ with $t\ge1$, and by Lemma~\ref{lam-1-ge-m} every Schur
function $s_\nu$ in the expansion of $h_t[h_{m+1}]$ has $\nu_1\ge m+1$. By the
Littlewood--Richardson rule, multiplying by the Schur-positive function
$h_{j-t}[A_m]$ preserves this property. Hence the pairing with $s_\lambda$ does
not change once $m\ge\lambda_1$.

\item \emph{Changing $n$ to $n+1$.} This adds the term
$$\frac1H\,h_{n+1}[A_m]=\frac1H\sum_{a_1+\dots+a_m=n+1}\ \prod_{i=1}^m h_{a_i}[h_i]
=\frac1H\sum_{a_1+\dots+a_m=n+1}h_{a_1}\prod_{i=2}^mh_{a_i}[h_i].$$
Each summand has the factor $h_{a_1}$, so every Schur function $s_\nu$ in its
expansion has $\nu_1\ge a_1$, and the summand has degree
$a_1+\sum_{i\ge 2}ia_i\ge a_1+2\sum_{i\ge 2}a_i=2(n+1)-a_1$. For the summand to
pair non-trivially with $s_\lambda$ we would need $a_1\le\lambda_1$ and
$2(n+1)-a_1\le |\lambda|$. Adding these inequalities gives
$n+1\le (|\lambda|+\lambda_1)/2$, which is impossible once
$n\ge\lfloor(|\lambda|+\lambda_1)/2\rfloor$.
\end{itemize}

Hence the coefficient is constant in this range, and it equals the coefficient
obtained in the limit $m,n\to\infty$:
$$\frac1H\cdot H[A_\infty]=H[-h_1]\,H[H-1]=H[H-1-h_1],$$
using $\frac1H=H[-h_1]$, $A_\infty=H-1$ and $H[X]\,H[Y]=H[X+Y]$.
\end{proof}

\begin{remark}
The two expressions for the stable value in Theorems~\ref{omitting-h-1}
and~\ref{thm-perp} agree. Since $H-h_1=1+(H-1-h_1)$ and $h_j[1]=1$,
\eqref{co-prod-for-h-n} gives $h_k[H-h_1]=\sum_{j\le k}h_j[H-1-h_1]$. Every
term of $H-1-h_1$ has degree at least $2$, so $h_j[H-1-h_1]$ has no terms of
degree less than $2j$; hence only $j\le\lfloor|\lambda|/2\rfloor$ contribute to
the coefficient of $s_\lambda$, in both $h_{\lfloor|\lambda|/2\rfloor}[H-h_1]$
and $H[H-1-h_1]$.
\end{remark}

\section{Partitions with few columns}\label{sec-few-columns}

In this section, we give a positive formula for the stable Foulkes
coefficient when $\lambda$ varies over partitions with a bounded number of
columns, that is, with $\lambda_1$ bounded. We first reduce the problem, for
$\lambda_1\le k$, to a multiplicity computation for $GL_k(\C)$, and then carry
this out explicitly for $k\le4$.

\smallskip
Throughout this section, we write
$$c(\lambda)=\langle s_\lambda,H[H-1-h_1]\rangle$$
for the stable Foulkes coefficient of $\lambda$ (Theorem~\ref{thm-perp}), and
$\lambda'$ for the conjugate partition.

\smallskip
Let $E=\sum_{n\ge0}e_n$. For a
finite-dimensional polynomial representation $U$ of $GL_k(\C)$ with character
$\mathrm{Ch}(U)$, the symmetric algebra $\mathrm{Sym}(U)=\bigoplus_n\mathrm{Sym}^n(U)$
and the exterior algebra $\Lambda(U)=\bigoplus_n\Lambda^n(U)$ have characters
$H[\mathrm{Ch}(U)]$ and $E[\mathrm{Ch}(U)]$.

\smallskip
We also use rational
representations: for weakly decreasing $\alpha\in\Z^k$, $W_\alpha(\C^k)$
denotes the irreducible representation with highest weight $\alpha$, so that
$W_\alpha(\C^k)\otimes{\det}^t=W_{\alpha+(t^k)}(\C^k)$ for $t\in\Z$.

\subsection{A reduction to \texorpdfstring{$GL_k(\C)$}{GLk(C)}}

\begin{proposition}\label{prop-reduction}
Let $k\ge1$, let $V=\C^k$, and let $\lambda$ be a partition with
$\lambda_1\le k$. Then $c(\lambda)$ is the multiplicity of $W_{\lambda'}(V)$ in
$$M_k=\bigotimes_{i=2}^{k}F_i\big(\Lambda^iV\big)\;\text{where}\;
F_i=\begin{cases}\mathrm{Sym}&\text{if } i \text{ is even},\\ \Lambda&\text{if } i\text{ is odd}.\end{cases}$$
\end{proposition}

\begin{proof}
Since $H[X+Y]=H[X]\,H[Y]$, we have $H[H-1-h_1]=\prod_{i\ge2}H[h_i]$; the
product makes sense because every term of $H[h_i]-1$ has degree at least $i$.

\smallskip
The involution $\omega$ is a ring automorphism and an isometry with
$\omega(s_\nu)=s_{\nu'}$, and $\omega(f[g])$ equals $f[\omega g]$ if $g$ is
homogeneous of even degree and $(\omega f)[\omega g]$ if $g$ is homogeneous of
odd degree \cite[Ch.~I, \S8]{macdonald}. Hence $\omega(H[h_i])$ equals
$H[e_i]$ for $i$ even and $E[e_i]$ for $i$ odd, and
$$c(\lambda)=\Big\langle s_{\lambda'},\ \prod_{i\ge2}\omega\big(H[h_i]\big)\Big\rangle.$$

\smallskip
Since $\ell(\lambda')=\lambda_1\le k$, this is the coefficient of
$s_{\lambda'}(x_1,\dots,x_k)$ in the specialization of the product to $k$
variables. Under this specialization $e_i\mapsto0$ for $i>k$, so the factors
with $i>k$ become $1$, while for $2\le i\le k$ the factor $H[e_i]$ or $E[e_i]$
becomes the character of $F_i(\Lambda^iV)$, because
$\mathrm{Ch}(\Lambda^iV)=e_i(x_1,\dots,x_k)$.
\end{proof}

\begin{corollary}\label{cor-three-columns}
Let $\lambda$ be a partition with $\lambda_1\le3$. Then $c(\lambda)=1$ if
$\lambda=(2^n)$ or $\lambda=(3,2^n)$ for some $n\ge0$, and $c(\lambda)=0$
otherwise.
\end{corollary}

\begin{proof}
Let $V=\C^3$. Taking $k=3$ in Proposition~\ref{prop-reduction}, we must
compute the multiplicity of $W_{\lambda'}(V)$ in
$M_3=\mathrm{Sym}(\Lambda^2V)\otimes\Lambda(\Lambda^3V)$. Since
$\Lambda^3V=\det$, we have $\Lambda(\Lambda^3V)=\mathbf 1\oplus\det$. Since
$\Lambda^2V\cong V^*\otimes\det$, we have
$\mathrm{Sym}^n(\Lambda^2V)\cong\mathrm{Sym}^n(V^*)\otimes{\det}^n=W_{(n,n,0)}(V)$.

\smallskip
Therefore
$$M_3\cong\bigoplus_{n\ge0}\Big(W_{(n,n,0)}(V)\oplus W_{(n+1,n+1,1)}(V)\Big),$$
which is multiplicity-free. The conjugates of $(n,n)$ and $(n+1,n+1,1)$ are
$(2^n)$ and $(3,2^n)$.
\end{proof}

\subsection{Partitions with at most four columns}

For a partition $\lambda$ with $\lambda_1\le4$, we write
$$\lambda=(4^{m_4}3^{m_3}2^{m_2}1^{m_1}),$$
so that
$$\lambda'=(m_1+m_2+m_3+m_4,\ m_2+m_3+m_4,\ m_3+m_4,\ m_4).$$

\begin{theorem}\label{thm-four-columns}
Let $\lambda=(4^{m_4}3^{m_3}2^{m_2}1^{m_1})$, and let
$\mathcal P_\lambda\subset\mathbb{R}^5$ be the polytope of points
$(t,\varepsilon_1,\varepsilon_2,\varepsilon_3,\varepsilon_4)$ satisfying
\begin{gather*}
0\le\varepsilon_i\le1,\qquad t+\varepsilon_1+\varepsilon_2+\varepsilon_3+\varepsilon_4\ge3,\\[4pt]
\varepsilon_1-\varepsilon_2=m_1,\qquad \varepsilon_3-\varepsilon_4=m_3,\qquad
\varepsilon_2-\varepsilon_3\le m_2,\qquad t+\varepsilon_4\le m_4.
\end{gather*}
Then $c(\lambda)$ is the number of lattice points of $\mathcal P_\lambda$, and is
given explicitly by the formula of Theorem~\ref{thmE}.
\end{theorem}

\begin{proof}
Let $V=\C^4$. Taking $k=4$ in Proposition~\ref{prop-reduction}, we must
compute the multiplicity of $W_{\lambda'}(V)$ in
$M_4=\mathrm{Sym}(\Lambda^2V)\otimes\Lambda(\Lambda^3V)\otimes\mathrm{Sym}(\Lambda^4V)$.
We decompose the three factors.
\begin{enumerate}\setlength{\itemsep}{6pt}
\item The symmetric algebra $\mathrm{Sym}(\Lambda^2V)$ is the multiplicity-free
sum of the $W_\alpha(V)$ over all partitions $\alpha$ with at most $4$ parts all
of whose columns have even length \cite[Ch.~I]{macdonald}; that is,
$\mathrm{Sym}(\Lambda^2V)\cong\bigoplus_{P\ge Q\ge0}W_{(P,P,Q,Q)}(V)$.
\item Since $\Lambda^3V\cong V^*\otimes\det$ and
$\Lambda^jV^*\cong\Lambda^{4-j}V\otimes{\det}^{-1}$,
$$\Lambda^j(\Lambda^3V)\cong\Lambda^j(V^*)\otimes{\det}^j\cong\Lambda^{4-j}V\otimes{\det}^{\,j-1}\qquad(0\le j\le4).$$
\item Since $\Lambda^4(V)\cong\mathrm{det}$, $\mathrm{Sym}(\Lambda^4V)=\bigoplus_{c\ge0}{\det}^c$.
\end{enumerate}

\smallskip
Writing $r=4-j$ and $t=j-1+c$, we obtain
$$M_4\cong\bigoplus_{P\ge Q\ge0}\ \bigoplus_{r=0}^{4}\ \bigoplus_{t\ge3-r}W_{(P,P,Q,Q)}(V)\otimes\Lambda^rV\otimes{\det}^t.$$

\smallskip
By Pieri's rule, $W_\alpha(V)\otimes\Lambda^rV$ is the multiplicity-free sum of
the $W_{\alpha+\varepsilon}(V)$ over all $\varepsilon\in\{0,1\}^4$ with
$|\varepsilon|=r$ for which $\alpha+\varepsilon$ is a partition. Hence the
multiplicity of $W_\mu(V)$, where $\mu=\lambda'$, is the number of
$(t,\varepsilon)\in\Z\times\{0,1\}^4$ with $t\ge3-|\varepsilon|$ such that
$\alpha=\mu-\varepsilon-(t^4)$ has the form $(P,P,Q,Q)$ with $P\ge Q\ge0$,
where $(t^4)=(t,t,t,t)$. Note that $\alpha+\varepsilon=\mu-(t^4)$ is always a
partition, since $\mu$ is, so the condition in Pieri's rule holds
automatically.

\smallskip
It remains to express the conditions on $\alpha$ in terms of $\lambda$. We have
$\alpha_i=\mu_i-\varepsilon_i-t$, and, by the formula for $\lambda'$ above,
$\mu_1-\mu_2=m_1$, $\mu_2-\mu_3=m_2$, $\mu_3-\mu_4=m_3$ and $\mu_4=m_4$.
Therefore
\begin{align*}
\alpha_1=\alpha_2 &\iff \varepsilon_1-\varepsilon_2=\mu_1-\mu_2=m_1,\\
\alpha_2\ge\alpha_3 &\iff \varepsilon_2-\varepsilon_3\le\mu_2-\mu_3=m_2,\\
\alpha_3=\alpha_4 &\iff \varepsilon_3-\varepsilon_4=\mu_3-\mu_4=m_3,\\
\alpha_4\ge0 &\iff t+\varepsilon_4\le\mu_4=m_4.
\end{align*}
These are the four conditions in the second line of the definition of
$\mathcal P_\lambda$, while the first line encodes
$\varepsilon\in\{0,1\}^4$ and $t\ge3-|\varepsilon|$ for lattice points. Hence
the multiplicity of $W_\mu(V)$ equals the number of lattice points of
$\mathcal P_\lambda$.

\medskip
For the explicit formula, the equations $\varepsilon_1-\varepsilon_2=m_1$ and
$\varepsilon_3-\varepsilon_4=m_3$ have no solutions in $\{0,1\}$ unless
$m_1,m_3\le1$. Otherwise, for each admissible $\varepsilon$, the integer $t$
ranges over the interval $3-|\varepsilon|\le t\le m_4-\varepsilon_4$, which
contains $(m_4-\varepsilon_4-2+|\varepsilon|)_+$ integers. The admissible
$\varepsilon$ are listed below; the condition $\varepsilon_2-\varepsilon_3\le m_2$
fails only for $\varepsilon=(1,1,0,0)$ when $m_2=0$.
\begin{center}
\renewcommand{\arraystretch}{1.3}
\begin{tabular}{c|c|c}
$(m_1,m_3)$ & admissible $\varepsilon$ & number of values of $t$\\\hline
$(0,0)$ & $(0,0,0,0)$ & $(m_4-2)_+$\\
        & $(1,1,0,0)$ & $m_4$ \ (only if $m_2\ge1$)\\
        & $(0,0,1,1)$ & $(m_4-1)_+$\\
        & $(1,1,1,1)$ & $m_4+1$\\\hline
$(1,0)$ & $(1,0,0,0)$ & $(m_4-1)_+$\\
        & $(1,0,1,1)$ & $m_4$\\\hline
$(0,1)$ & $(0,0,1,0)$ & $(m_4-1)_+$\\
        & $(1,1,1,0)$ & $m_4+1$\\\hline
$(1,1)$ & $(1,0,1,0)$ & $m_4$
\end{tabular}
\end{center}
\smallskip
\noindent Summing each row gives the formula of Theorem~\ref{thmE}.
\end{proof}

\begin{remark}
Theorem~\ref{thm-four-columns} has several immediate consequences. If
$\lambda_1\le4$, then $c(\lambda)=0$ whenever $\lambda$ has at least two parts
equal to $1$ or at least two parts equal to $3$; $c(\lambda)$ grows linearly in
$m_4$; and it depends on $m_2$ only through whether $m_2\ge1$. Setting $m_4=0$
recovers Corollary~\ref{cor-three-columns}. For example, for $\lambda=(4,2,2)$
we have $m_4=1$, $m_2=2$, $m_1=m_3=0$, so $c(\lambda)=2+0+0+1=3$, in agreement
with the computation in Section~\ref{sec-sage}.
\end{remark}

\begin{remark}
The same approach applies to partitions with more columns, but becomes less
explicit. For $k=5$ and $V=\C^5$, each factor of $M_5$ is, up to a power of
$\det$, a symmetric or exterior algebra of $V$, $V^*$, $\Lambda^2V$ or
$\Lambda^2V^*$, and these
are multiplicity-free by classical results of Littlewood; $c(\lambda)$ is then
a sum of products of Littlewood--Richardson coefficients. For $k=6$, the
finite-dimensional factor $\Lambda(\Lambda^3\C^6)$ must be decomposed
separately. For $k\ge7$, the factor $\mathrm{Sym}(\Lambda^4\C^k)$ is not
multiplicity-free, and computing its decomposition is itself a plethysm
problem.
\end{remark}

\section{Computing stable Foulkes coefficients with SageMath}\label{sec-sage}

\begin{definition}
For a partition $\lambda$, the common value of
$\langle h_n[h_m],s_{\lambda[mn]}\rangle$ for $m,n\ge|\lambda|$ is called the
\emph{stable Foulkes coefficient} of $\lambda$.
\end{definition}

In this section, we compute a few stable Foulkes coefficients using the
formula of Theorem~\ref{thmA}, which follows from Theorem~\ref{proof1} and
Theorem~\ref{equality-of-vector-partitions}:
\begin{equation}
    \lim_{m,n\to\infty}\langle h_n[h_m],s_{\lambda[mn]}\rangle=\sum_{\sigma\in S_{\ell(\lambda)}}\sgn(\sigma)\,p^{\geq 2}\big(\lambda_1+\sigma(1)-1,\lambda_2+\sigma(2)-2,\dots,\lambda_{\ell(\lambda)}+\sigma(\ell(\lambda))-\ell(\lambda)\big).
\end{equation}
\begin{itemize}
    \item \emph{$\lambda=(k)$ is a one-row partition:}
    $\lim_{m,n\to\infty}\langle h_n[h_m],s_{(mn-k,k)}\rangle=p^{\ge 2}(k)$, the number
    of partitions of $k$ with all parts at least $2$.
    \item \emph{$\lambda=(2,2)$:} For all $m,n\geq 4$,
    \begin{equation*}
        \langle h_n[h_m],s_{(mn-4,2,2)}\rangle=p^{\ge 2}(2,2)-p^{\ge 2}(3,1)=3-2=1.
    \end{equation*}
    \item \emph{$\lambda=(4,2,2)$:} For all $m,n\geq 8$,
    \begin{align*}
        \langle h_n[h_m],s_{(mn-8,4,2,2)}\rangle&=\sum_{\sigma\in S_3}\sgn(\sigma)\,p^{\geq 2}\big(4+\sigma(1)-1,2+\sigma(2)-2,2+\sigma(3)-3\big)\\
        &=\sum_{\sigma \in S_3} \sgn(\sigma)\, p^{\ge 2}\big(3 + \sigma(1), \sigma(2), \sigma(3) - 1\big)\\
        &= p^{\ge 2}(4, 2, 2) + p^{\ge 2}(5, 3, 0) + p^{\ge 2}(6, 1, 1) \\
        &\quad-p^{\ge 2}(5, 1, 2) - p^{\ge 2}(4, 3, 1) - p^{\ge 2}(6, 2, 0)\\
        &=57+22+23-36-44-19=3.
    \end{align*}
    \item \emph{$\lambda=(1^3)$:} For all $m,n\ge 3$,
    \begin{align*}
        \langle h_n[h_m],s_{(mn-3,1^3)}\rangle&=\sum_{\sigma\in S_3}\sgn(\sigma)\,p^{\ge 2}\big(\sigma(1),\sigma(2)-1,\sigma(3)-2\big)\\
        &=p^{\ge 2}(1,1,1) - p^{\ge 2}(2,0,1) - p^{\ge 2}(1,2,0)\\
        &\quad- p^{\ge 2}(3,1,-1) + p^{\ge 2}(2,2,-1) + p^{\ge 2}(3,0,0)\\
        &=1-1-1-0+0+1=0.
    \end{align*}
\end{itemize}

\begin{remark}
The values of $p^{\ge 2}(\cdot)$ above were computed in SageMath with the
following code.
\end{remark}
\begin{lstlisting}[language=Python]
def count_vector_partitions_min_size_2(alpha):
    """
    Number of vector partitions of alpha in which every part
    has size (sum of components) at least 2.
    """
    count = 0
    for partition in VectorPartitions(alpha):
        if all(sum(part) >= 2 for part in partition):
            count += 1
    return count
\end{lstlisting}

Alternatively, the stable value can be computed as
$\langle h_{\lfloor|\lambda|/2\rfloor}[H-h_1],s_\lambda\rangle$ (Theorem~\ref{omitting-h-1}).
\begin{itemize}
    \item \emph{$|\lambda|=1$:} The stable value is $\langle h_0[H-h_1],s_\lambda\rangle=\langle 1,s_\lambda\rangle=0$.
    \item \emph{$|\lambda|=2,3$:} The stable value is
    $$\langle h_1[H-h_1],s_\lambda\rangle=\langle H-h_1,s_\lambda\rangle
    =\begin{cases}
        1 & \text{if } \lambda=(2) \text{ or } (3),\\
        0 & \text{otherwise.}
    \end{cases}$$
    \item \emph{$|\lambda|=4,5$:} The stable value is
    $\langle h_2[H-h_1],s_\lambda\rangle$, which is zero when $\ell(\lambda)>2$.
    For $\ell(\lambda)\le 2$, the possibilities are
    $\lambda\in\{(4),(3,1),(2,2),(5),(4,1),(3,2)\}$. Since only terms of degree
    at most $|\lambda|$ contribute,
    \begin{align*}
    \langle h_2[H-h_1], s_\lambda\rangle
      &= \langle h_2[h_0+h_2+h_3+h_4], s_\lambda\rangle \quad\text{if } |\lambda| = 4,\\
    \langle h_2[H-h_1], s_\lambda\rangle
      &= \langle h_2[h_0+h_2+h_3+h_4+h_5], s_\lambda\rangle \quad\text{if } |\lambda| = 5,
    \end{align*}
    which gives
    \begin{align*}
    &\langle h_2[H-h_1], s_{(4)}\rangle=2, && \langle h_2[H-h_1], s_{(3,1)}\rangle=0, && \langle h_2[H-h_1], s_{(2,2)}\rangle=1,\\
    &\langle h_2[H-h_1], s_{(5)}\rangle=2, && \langle h_2[H-h_1], s_{(4,1)}\rangle=1, && \langle h_2[H-h_1], s_{(3,2)}\rangle=1.
    \end{align*}
\end{itemize}
These values were computed with the following SageMath code:
\begin{lstlisting}[language=Python]
S = SymmetricFunctions(QQ)
s = S.schur()
h = S.complete()
print(s(h[2](h[0]+h[2]+h[3]+h[4])))
print(s(h[2](h[0]+h[2]+h[3]+h[4]+h[5])))
\end{lstlisting}
whose output is
\begin{lstlisting}[numbers=none]
s[] + s[2] + s[2, 2] + s[3] + s[3, 2] + 2*s[4] + s[4, 1] + 2*s[4, 2] + s[4, 3] + s[4, 4] + s[5] + s[5, 1] + s[5, 2] + 2*s[6] + s[6, 1] + s[6, 2] + s[7] + s[8]

s[] + s[2] + s[2, 2] + s[3] + s[3, 2] + 2*s[4] + s[4, 1] + 2*s[4, 2] + s[4, 3] + s[4, 4] + 2*s[5] + s[5, 1] + 2*s[5, 2] + s[5, 3] + s[5, 4] + 2*s[6] + 2*s[6, 1] + 2*s[6, 2] + s[6, 3] + s[6, 4] + 2*s[7] + s[7, 1] + s[7, 2] + 2*s[8] + s[8, 1] + s[8, 2] + s[9] + s[10]
\end{lstlisting}

\section*{Acknowledgments}
I thank my supervisor, Prof.~Amritanshu Prasad, for fruitful discussions. I
also thank Prof.~Chris Bowman and Prof.~Rowena Paget for their encouraging
feedback.
\section*{Declaration on the use of AI tools}
In preparing this manuscript, the author used Claude (Anthropic) to polish sentences, to check references, and to identify typographical errors. No mathematical content was generated by AI tools. The author reviewed and verified the entire manuscript and takes full responsibility for its contents.

\end{document}